\documentclass[12pt]{amsart}

\usepackage{pb-diagram} 

\usepackage{pdfsync}
\usepackage{xcolor}

\usepackage{geometry}
\usepackage{lscape}
\usepackage{graphicx}
\usepackage{epstopdf}    
\usepackage[matrix,arrow,curve,frame]{xy}
\usepackage{amsmath,amsthm,amssymb,enumerate}
\usepackage{latexsym}
\usepackage{amscd}
\usepackage{amssymb}

\numberwithin{equation}{section}

\theoremstyle{plain}  
\newtheorem{thm}[equation]{Theorem}
\newtheorem{prop}[equation]{Proposition}
\newtheorem{claim}[equation]{Claim}

\newtheorem{cor}[equation]{Corollary}

\theoremstyle{definition}  
\newtheorem{defn}[equation]{Definition}

\newtheorem{remark}[equation]{Remark}

\newcommand{\ra}{\rightarrow}

\newcommand{\R}{\mathbb R}

\newcommand{\Z}{\mathbb Z}
\newcommand{\C}{\mathbb C}

\newcommand{\E}{\mathbb E}
\newcommand{\ER}{\mathbb{ER}}

\DeclareMathOperator{\id}{id}

\begin{document}
\pagestyle{plain}

\title
{The $ER(n)$-cohomology of $BO(q)$, and real Johnson-Wilson orientations for vector bundles.}
\author{Nitu Kitchloo}
\author{W. Stephen Wilson}
\address{Department of Mathematics, Johns Hopkins University, Baltimore, USA}
\email{nitu@math.jhu.edu}
\email{wsw@math.jhu.edu}
\thanks{Nitu Kitchloo is supported in part by NSF through grant DMS
  1307875.}

\date{\today}


{\abstract

Using the Bockstein spectral sequence developed 
previously by the authors, 
we compute the ring $ER(n)^\ast(BO(q))$ explicitly. 
We then use this calculation to show that the 
ring spectrum $MO[2^{n+1}]$ is $ER(n)$-orientable (but 
not $ER(n+1)$-orientable), where $MO[2^{n+1}]$ is defined as the 
Thom spectrum for the self map of $BO$ given by 
multiplication by $2^{n+1}$.}
\maketitle


\section{Introduction}

The $p=2$ Johnson-Wilson theory, $E(n)$, is a well known complex 
oriented cohomology theory with coefficients given by:
\[
E(n)^* = \Z_{(2)}[v_1,v_2, \ldots, v_{n-1},v_n^{\pm 1}],
\]
where $|v_i|=-2(2^i-1)$. The orientation map $MU \longrightarrow 
E(n)$ can be constructed in the category of $\Z/2$-equivariant 
$MU$-module spectra, with the $\Z/2$-action on $MU$ representing 
complex conjugation $c$. As shown by Hu and Kriz in \cite{HK}, 
the spectra $E(n)$ are suitably complete so that the homotopy 
fixed point spectra $E(n)^{h\Z/2}$ agree with $E(n)^{\Z/2}$. 
We call this ring spectrum the `real' Johnson-Wilson theory, $ER(n)$. 

The main result is very straightforward and simple to state.
There are Conner-Floyd Chern classes, $\hat{c}_k \in ER(n)^*(BO(q))$,
for $0 < k \le q$, and corresponding complex conjugate classes, $\hat{c}_k^*$,
with their degrees given by $-k2^{n+2}(2^{n-1}-1)$.

\begin{thm}
\label{main}
There is a canonical isomorphism:
$$
 ER(n)^\ast(BO(q)) \simeq
 ER(n)^\ast[[\hat{c}_1, \ldots, \hat{c}_q]]/(
\hat{c}_1 - \hat{c}_1^\ast, \ldots, 
\hat{c}_q - \hat{c}_q^\ast ).
$$
\end{thm}

This follows, in a non-trivial way, from the similar result
for $E(n)$, (\cite{WSW:BO}, \cite{KY}, 
\cite{Kriz},
\cite{RWY}). 
In the process, we give a nice description of $ER(n)^*$, which
has some value in its own right, and its computation.  We
show that our Bockstein spectral sequence from $E(n)^*$ to
$ER(n)^*$ takes place in a particularly nice category over
operations.  This allows us to use the established Landweber flatness
of $E(n)^*(BO(q))$ to accomplish the above result.

\begin{cor}
\label{cor}
There is a canonical isomorphism:
$$
 ER(n)^\ast(BO) \simeq
ER(n)^\ast[[\hat{c}_1, \hat{c}_2,\ldots ]]/(
\hat{c}_1 - \hat{c}_1^\ast ,
\hat{c}_2 - \hat{c}_2^\ast ,\ldots
).
$$
\end{cor}

\begin{remark}
When $n=1$, $ER(1) = KO_{(2)}$, and this gives a description of
$KO_{(2)}^*(BO)$ that is quite different from the usual one
from \cite{AndersonKOBO} (and later \cite{AS}), where $KO^0(BO)$
is given by representations as $RO(O)\hat{}$. 
Our work is also clearly connected to \cite{HaraK}, where $KO^*(BO(q))$
is computed.
\end{remark}

We take a brief excursion into vector bundles with an
$ER(n)$ orientation.  Although clearly not the complete
answer, we prove the following theorem.

Define the bundle $2^k \xi$ over $BO$ to be the pullback of the 
universal bundle $\xi$ along the multiplication 
by $2^k$-map $[2^k] : BO \longrightarrow BO$. Let $MO[2^k]$ 
denote the Thom spectrum of the bundle $2^k\xi$.

\begin{thm}
\label{orient}
$MO[2^{n+1}]$ admits a canonical $ER(n)$-orientation. In particular, 
given any real vector bundle $\eta : V \rightarrow B$, the 
bundle $2^{n+1}\eta$ admits an $ER(n)$-orientation. 
$MO[2^{n+1}]$ does not admit an $ER(n+1)$-orientation. 
\end{thm}

In Section 2 we set up our Bockstein spectral sequence
in a slightly different way than we did in 
\cite{NituP} and \cite{NituP2}.  We describe the behavior of the spectral 
sequence for $ER(n)^*(pt)$ in Section 3 and show all the modules and
differentials are in our special category in Section 4.
The computation of $ER(n)^*(BO(q))$ is completed in Section 5
and the techniques are applied to orientations in Section 6.

\section{The Bockstein Spectral Sequence:}

In the paper \cite{Nitufib} we show that the homotopy 
ring of $ER(n)$ is a subquotient of a ring:
\[ \Z_{(2)}[x, \hat{v}_1, \hat{v}_2, \ldots, \hat{v}_{n-1},v_n^{\pm 1}], \] 
where $x$ is an element of $\pi_\ast ER(n)$ in 
degree $\lambda = \lambda(n) = 2^{2n+1}-2^{n+2}+1 = 2(2^n-1)^2-1$. 
The classes $\hat{v}_k$ for $k \leq n$ exist 
in $ER(n)^\ast$ in degree $2^{n+2}(2^{n-1}-1)(2^k-1)$, 
and map to the respective classes of the same name 
$\hat{v}_k = v_k v_n^{-(2^n-1)(2^k-1)}$ 
under the canonical map from $ER(n)^\ast$ to $E(n)^\ast$, where
$v_0 = 2$. 

In \cite{Nitufib} we also construct a fibration of spectra:
\[
\Sigma^\lambda ER(n) \ra ER(n) \ra E(n)
\]
with the first map given by multiplication with the class $x$. 
The class $x$ is a $2$-torsion class of exponent $2^{n+1}-1$, 
and consequently this fibration leads to a convergent 
Bockstein spectral sequence described explicitly by the following theorem: 

\begin{thm} (compare with \cite[Theorem 4.2]{NituP} 
\footnote{In {\em loc. cit.} we consider the untruncated 
version of the spectral sequence converging to zero.})
\label{bss} \mbox{ }
\begin{enumerate}[(i)]
\item
For $X$ a spectrum,
the above fibration yields a first and fourth quadrant 
spectral sequence of $ER(n)^\ast$-modules, $\mbox{E}_r^{i, j}(X) 
\Rightarrow ER(n)^{j-i}(X)$. The differential $d_r$ has 
bi-degree $(r, r+1)$ for $r \geq 1$.
\item
The $\mbox{E}_1$-term is given by: $\mbox{E}_1^{i,j}(X) = 
E(n)^{i\lambda +j-i}(X)$, with 
\[ d_1(z) = v_n^{-(2^n-1)}(1-c)(z), \,  
\text{ where }  \, 
c(v_i)= -v_i.
\]
The differential $d_r$ 
increases cohomological degree by $1+r\lambda$
between the appropriate 
sub-quotients of $E(n)^\ast(X)$. 
\item
For $r < 2^{n+1}$, the targets of the differentials,
$d_r$, represent the image of $x^r$-torsion generators 
of\break $ER(n)^*(X)$ inside $E(n)^\ast(X)$. 
\item
$\mbox{E}_{2^{n+1}}(X) = \mbox{E}_{\infty}(X)$, which is 
described as follows: Filter $M = ER(n)^*(X)$ by $M_i = x^i  M$ so that:
\[
M = M_0 \supset M_1 \supset M_2 \supset \cdots \supset M_{2^{n+1}-1} = \{0 \}.
\]
Then $\mbox{E}_\infty^{r,\ast}(X)$ is canonically isomorphic to $M_r/M_{r+1}$.  
\item
The following are
all vector
spaces over $\Z/2$:
$$
M_i/M_j, \quad j \ge i > 0, 
\mbox{ and }
\mbox{E}_r^{i,j}(X) 
\quad r > 1
\mbox{ and }
i > 0.
$$
\item 
$d_r(ab) = d_r(a) b + c(a) d_r(b)$. 
In particular, if $c(z) = z \in \mbox{E}_r(X)$, then $d_r(z^2) = 0$, $r > 1$.

\end{enumerate}
\end{thm}

\begin{remark}
When $X$ is a space,
notice that there is a canonical 
class $ y \in \mbox{E}_1^{1,-\lambda +1}(X)$ that 
corresponds to the unit element under the identification 
of $\mbox{E}_1^{1,-\lambda+1}(X) = E(n)^0(X)$. 
This class represents the element $x \in ER(n)^{-\lambda}$, 
and is therefore a permanent cycle. We may use this 
class to simplify the notation.
$$
\mbox{E}_1^{\ast, \ast}(X) = \mbox{E}_1^{0,\ast}(X)[y] = E(n)^\ast(X)[y], 
\quad |y|=(1,-\lambda +1). 
$$
$$
d_1(z) = y \, v_n^{-(2^n-1)}(1-c)(z), \,  
\text{ where }  \, 
v_n \in \mbox{E}_1^{0,-2(2^{n}-1)}.
$$
\end{remark}

\begin{remark}
The formal structure of differentials in the 
Bockstein spectral sequence appears to be identical to that of the 
Borel homology spectral sequence for the real 
spectrum $\mathbb{ER}(n)$ described in \cite{HK}. 
Since they both converge to the same object, one may be 
tempted to conclude that {\em they must be the same spectral sequence}. 
However, these spectral sequences encode very different 
types of information, and to our knowledge there is no direct 
way to identify them. Note also that we have 
formulated the spectral sequence in cohomology. 
Of course, there is a corresponding homology 
Bockstein spectral sequence $\mbox{E}^r(X)$ converging to $ER(n)_\ast(X)$. 
\end{remark}

\section{The spectral sequence for $X=pt$ }

In this section, 
we organize the ring $ER(n)^\ast$ in terms of the 
Bockstein spectral sequence, $\mbox{E}_\ast(X)$, for $X$ a point. 
The theorem is not a new calculation of $ER(n)^*$, 
which is already known from \cite{HK}, and we use their results.  
This approach could not give us a new calculation because the
very fibration the spectral sequence depends on comes out of
a calculation from \cite{HK}.
What the theorem is about is the behavior of 
the Bockstein spectral sequence.

Our description of $ER(n)^*$ is much nicer than the usual descriptions
and gives us access to information we need for the computation of
$ER(n)^*(BO(q))$.
We present the spectral sequence as a sum of very nice modules over the
ring 
$R_n = \Z_{(2)}[\hat{v}_1,\ldots,\hat{v}_{n-1}]$.
In our description, we need the ideals
$I_j=(2,\hat{v}_1,\ldots,\hat{v}_{j-1})$ for $0< j \le n$, and $I_0 = (0)$.

Although our modules are not always cyclic, they will all be
associated with particularly nice elements, namely $  v_n^k \, y^m$.
It is important to observe that the modules $I_i R_n/I_j$, which
naturally come up often in the proof, 
are trivial whenever $0 < i \le j$.
We set $R_n/I_{n+1} = 0$.

Our description is as follows:

\begin{thm} \label{structure}
In the 
spectral sequence
$\mbox{E}_r(pt) 
\Rightarrow ER(n)^{*}$, 
\begin{enumerate}[(i)]

\item
$$
\mbox{E}_1  \simeq
\Z_{(2)}[y, \hat{v}_1, \hat{v}_2, \ldots, \hat{v}_{n-1},v_n^{\pm 1}].
$$
That is,
$$
\mbox{E}_{1}^{m,*}=
\Z_{(2)}[\hat{v}_1, \hat{v}_2, \ldots, \hat{v}_{n-1},v_n^{\pm 1}] \quad 
\mbox{on} \quad y^m.
$$

\item
The only non-zero differentials are generated by
$$
d_{2^{k+1}-1} (v_n^{-2^k}) = \hat{v}_k \, y^{2^{k+1}-1} \, v_n^{-2^{n+k}}
\quad \mbox{ for } \quad
0 \le k \le n.
$$

\item
$
\mbox{E}_{2^k}^{*,*}=
\mbox{E}_{2^{k+1}-1}^{*,*},
$
for $0 \le k \le n$, and 
$
\mbox{E}_{2^{n+1}}^{*,*}=
\mbox{E}_{\infty}^{*,*}.
$

\item
For $0 \le j <  k \le n+1$, 
$$
\mbox{E}_{2^k}^{m,*}=
R_n[ v_n^{\pm 2^k}]/I_j \bigoplus_{j < i < k}
I_i R_n[v_n^{\pm  2^{i+1}}] v_n^{2^i }/I_j 
\quad \mbox{ on } \quad y^m.
$$
when $2^j -1 \le m < 2^{j+1} -1$.

\item
For $0 < k \le n+1$ and
$ 2^k-1 \le m$
$$
\mbox{E}_{2^k}^{m,*}=
R_n[ v_n^{\pm 2^k}]/I_k 
\quad \mbox{ on } \quad  y^m. 
$$

\end{enumerate}
\end{thm}

\begin{remark}
Note two things.  First, when $m=0$ we must have $j=0$ and $I_j = 0$.
Second, when $k=n+1$, and $2^{n+1} - 1 \le m$,
$\mbox{E}_{\infty}^{m,*}= 0$.
\end{remark}

\begin{remark}
To put this all on familiar territory, recall that $E(1)=KU_{(2)}$ and
$ER(1)=KO_{(2)}$.
From above, we have
$$
\mbox{E}_1^{*,*}\simeq 
\Z_{(2)}[y,v_1^{\pm 1}].
$$
Recall that $d_1$ is generated by $d_1(v_1^{-1})=2 y v_1^{-2}$.
Keeping in mind that $R_1 = \Z_{(2)}$ and $R_1/I_1 = \Z/2$,
we get
$$
\mbox{E}_2^{0,*}\simeq 
\Z_{(2)}[v_1^{\pm 2}]
\quad \mbox{ and } \quad
\mbox{E}_2^{m,*}\simeq 
\Z/2[v_1^{\pm 2}]y^m, \quad m > 0.
$$
We have $d_3$ is generated by $d_3(v_1^{-2}) = \hat{v}_1 y^3 v_1^{-4}$.
Recalling that $\hat{v}_1 = v_1 v_1^{-(2^1-1)(2^1-1)} = 1$, this is
really
$d_3(v_1^{-2}) =  y^3 v_1^{-4}$.

We get, with $I_1 R_1 = 2 \Z_{(2)}$, 
$\mbox{E}_{\infty}^{m,*}$ is, for $m=0$,
$$
\Z_{(2)}[v_1^{\pm 4}] 
$$
$$
2 \Z_{(2)}[v_1^{\pm 4}] v_1^2 
$$
and, for $0 < m < 3$,
$$
\Z/2[v_1^{\pm 4}] \quad \mbox{ on } \quad   y^m.
$$
This is our
usual homotopy of $KO_{(2)}$. 
\end{remark}

\begin{remark}
A more illuminating and less familiar example, except for those familiar
with 
\cite{Nitufib},
\cite{NituER2}, 
\cite{NituP}, 
\cite{NituP2}, and \cite{Kitch-Wil-split}, (or who think $ER(2)$ is
really $TMF(3)$), 
is the case of $ER(2)^*$.  Looking briefly
at just the final answer, we have
$\mbox{E}_{\infty}^{0,*}$ is
$$
R_2[v_2^{\pm 8}] =  \Z_{(2)}[\hat{v}_1,v_2^{\pm 8}] ,
$$
$$
I_1 R_2[v_2^{\pm 4}] = 2 \Z_{(2)}[\hat{v}_1,v_2^{\pm 4}] 
v_2^2 ,
$$
$$
I_2 R_2[v_2^{\pm 8}] = (2,\hat{v}_1) \Z_{(2)}[\hat{v}_1,v_2^{\pm 8}] 
v_2^4 .
$$

This is already where most of the new interesting stuff happens.
Recall that this is periodic of order 48 on $v_2^8$, so we can
simplify by working with degrees $\Z/(48)$.  In this case, 
$\hat{v}_1 = v_1 v_2^{-(2^2-1)(2^1-1)} = v_1 v_2^{-3}$ and so
is in degree 16 (or -32) and we have always called this element $\alpha$. 

The generators $2 v_2^2$ and $2 v_2^6$ are in degrees -12 and -36 and
go by our names of $\alpha_3$ and $\alpha_1$ respectively.  The more
interesting elements are the $2 v_2^4$ and $\hat{v}_1 v_2^4$ in 
degrees -24 and -8, and known to us as $\alpha_2$ and $w$ respectively.
The interesting part here is that these two elements come from
the first part of
$(2,\hat{v}_1) \Z_{(2)}[\hat{v}_1,v_2^{\pm 8}] $, so that $\alpha \alpha_2
= 2 w$, a relation well-known to us, but easily visible here.

The element $y$ is our $x$ of degree -17.

We see, from the theorem, that for 
$\mbox{E}_{\infty}^{m,*}$, $m > 0$,
we must have $0 < j < i < n+1 = 3$, so $j=1$ and $i=2$, giving us
$I_2 R_2[v_n^{\pm {8}}]v_2^4/I_1$, 
or $\hat{v}_1 \Z/2 [\hat{v}_1,v_n^{\pm 8}]v_2^4$ on $  y^m$ 
 and 
$R_2[v_2^{\pm 8}]/I_1 = 
\Z/2[\hat{v}_1,v_2^{\pm 8}]$ on $y^m $, $m = 1 \mbox{ or } 2$.
Finally, there is the $R_2[v_2^{\pm 8}]/I_2 = \Z/2[v_2^{\pm 8}]$ 
on $y^m$ for $3 \le m < 7$.

\end{remark}

\begin{remark}
As $n$ increases, the complexity of $ER(n)^*$ also increases, but it 
really only does it one step at a time.
The only really new thing each $n$ adds is 
$$
I_n R_n[v_n^{\pm 2^{n+1}}]/I_0 
= (2,\hat{v}_1,\ldots,\hat{v}_{n-1})\Z_{(2)}[\hat{v}_1,
\ldots,\hat{v}_{n-1},v_n^{\pm 2^{n+1}}] v_n^{2^n}.
$$
We saw above how this affected $ER(2)^*$.  Looking just at this
for $n=3$, we have 
$$
I_3 R_3[v_3^{\pm 16}]v_3^8/I_0 = (2,\hat{v}_1,\hat{v}_{2})\Z_{(2)}[\hat{v}_1,
\hat{v}_{2},v_3^{\pm 16}]v_3^8. 
$$
From this we would get 3 new generators, generalizing those we
had in $ER(2)^*$:
$$
A = 2 v_3^8, \quad B = \hat{v}_1 v_3^8, \quad \mbox{ and } 
C = \hat{v}_2 v_3^8.
$$
This leads to all new, but obvious, relations:
$$
\hat{v}_1 A = 2 B, \quad \hat{v}_2 A = 2 C, \quad \hat{v}_2 B = \hat{v}_1 C,
$$
and, of course:
$$
\hat{v}_1 \hat{v}_2 A = 2 \hat{v}_2 B =  2 \hat{v}_1 C.
$$
This is just a sample, but the generators and relations are
easily read off from our description of $ER(n)^*$.
\end{remark}

\begin{remark}
Note that for the final differential we have
$$
d_{2^{n+1}-1} (v_n^{-2^n}) = \hat{v}_n \, y^{2^{n+1}-1} \, v_n^{-2^{n+n}}.
$$
But
$\hat{v}_n = v_n v_n^{-(2^n-1)(2^n-1)}$ so
$$
 \hat{v}_n
 v_n^{ -2^{n+n}}
=
 v_n v_n^{-(2^n-1)(2^n-1)}
 v_n^{-2^{n+n}}
$$
Recall that $v_n^{2^{n+1}}$ is the periodicity element.  
The exponent, modulo $2^{n+2}$, of $v_n$ in $\hat{v}_n$
is just 
$$
1 - (2^n-1)(2^n-1) -2^{n+n} = 1-2^{2n}+2^{n+1} -1 -2^{2n} = 2^{n+1}.
$$
So, $d_{2^{n+1}-1}$ takes the set of generators $v_n^{b2^{n+1}+2^n}$
to
the set
$v_n^{a2^{n+1}}\, y^{2^{n+1}-1}$.
\end{remark}

\begin{proof}[Proof of Theorem \ref{structure}]
The $\mbox{E}_1$ term is $E(n)^*[y]$, or
$
 \Z_{(2)}[y,v_1,v_2, \ldots, v_{n-1},v_n^{\pm 1}].
$
However, this representation is complicated by the fact 
that $d_1(v_k) \ne 0$.
We can replace the $v_k$ with
the equivalent
classes $\hat{v}_k = v_k v_n^{-(2^n-1)(2^k-1)}$, for $k < n$, 
which, by \cite{HK},
are permanent cycles.  Also from \cite{HK}, we use the fact that
$\hat{v}_k v_n^{a2^{k+1}}$ is precisely $y^{2^{k+1}-1}$ torsion,
which gives us the stated differentials. 

Since $d_1$ is determined by $d_1(v_n^{-1}) = 2 y v_n^{-2^n}$ and  
$d_1(v_n^2)=0$, we get $d_1(v_n^{2i-1}) = 2 y v_n^{-2^n+2i}$. It is
now easy to read off the above $\mbox{E}_2$ and this begins our
induction on $k$ by establishing $k=1$.

Assume our description holds
for some $k$, $1 \le k \le n$.  Because $y$ is a permanent cycle, $d_r$
commutes with $y$.

We need to study $d_{2^{k+1}-1}$ on $\mbox{E}_{2^{k+1}-1}^{m,*}$ for
each term of our description of the spectral sequence.

We first observe that the differential is trivial on all of
the $I_i R_n[v_n^{\pm 2^{i+1}}]/I_j$ terms in our answer.  
The image of the 
differential must lie in the group with $m \ge 2^{k+1}-1$, but
this group is always $R_n[v_n^{\pm 2^k}]/I_k$.  Since $i < k$, the image of
$I_i R_n[v_n^{\pm 2^{i+1}}]/I_j$ 
must be zero because every element of the source is
a multiple of an element of $I_i$, and all such elements are
zero in $I_k$.

We now need to compute the differentials on the terms (with various $m$):
$$
R_n[ v_n^{\pm 2^k}]/I_j 
\quad \mbox{ on } \quad  y^m \qquad 0 \le j \le k.
$$
We rewrite these as 
$R_n[ v_n^{\pm 2^{k+1}}]/I_j $ on $1$ and $v_n^{2^k}$.
The differential is trivial on the first part and maps
the second part to the first with a higher $m$ so that
the image is $\hat{v}_k R_n[v_n^{2^{k+1}}]/I_k$.   
Thus the cokernel for $m \ge 2^{k+1}-1$ is 
$R_n[ v_n^{\pm 2^{k+1}}]/I_{k+1}$  and
the
kernel is 
$I_k R_n[ v_n^{\pm 2^{k+1}}]/I_j $ on  $v_n^{2^k}$
as advertised.
Note that when $j=k$ this is trivial.
\end{proof}

\section{Landweber flatness in the bigraded setting}

Let us now identify more structure on this spectral sequence in
the special case above for $X$ a point. 
The spectral sequence is a spectral sequence of modules over
$ER(n)^*$.  Define the map $\psi : E(n)^* \rightarrow ER(n)^*$
by $\psi(v_k) = \hat{v}_k$ for $k \le n$. 
This map multiplies degrees by $(1-\lambda)/2$.
We can now view the
spectral sequence as being over $E(n)^*$.  We want more.  Our
aim in this section is to
show that the Bockstein spectral sequence for a point lives in the 
category of $E(n)_\ast E(n)$-comodules that are finitely 
presented as $E(n)_\ast$-modules. Since $E(n)^*$ is a particularly
nice ring, this last condition is never a problem for us.

Let $BP(n)$ denote the localization of $BP$ given 
by inverting $v_n$, i.e. $BP[v_n^{-1}]$. 
We begin by setting things up as 
$BP(n)_\ast BP(n)$-comodules, but by 
\cite[Theorem C]{HS}, this is equivalent to working with 
$E(n)_\ast E(n)$-comodules.

There is 
however a subtlety with grading we encounter when 
working in this setting. To explain this, let $\mathbb{BP}(n)$ 
denote the 
localization $[v_n^{-1}] \mathbb{BPR}$ obtained from the 
real (RO($\Z/2$))-graded Brown-Peterson spectrum on 
inverting the class $v_n$ in degree $(2^n-1)(1+\alpha)$. 
The (bigraded) homotopy groups of $\mathbb{BP}(n)$ 
and $\mathbb{BP}(n) \wedge \mathbb{BP}(n)$ can be 
computed using the Borel homology spectral sequence, 
and the Atiyah-Hirzebruch spectral sequence respectively. 
In particular, one observes:
\[ \pi_{(\ast, \ast)} (\mathbb{BP}(n) 
\wedge \mathbb{BP}(n)) \simeq \pi_{(\ast, \ast)}
(\mathbb{BP}(n))[t_1, t_2, \ldots, ] [v_n^{-1}] 
\simeq [v_n^{-1}] \pi_{(\ast,\ast)} (\mathbb{BPR})[t_1, t_2, 
\ldots, ] [v_n^{-1}], \]
where all generators are in diagonal bidegree, and the 
left (respectively, right) inverse of $v_n$ denotes the 
inversion of the map given by multiplication by $v_n$ 
on the respective $\mathbb{BPR}$ factor. 
We may multiply any diagonal class $z$ in bidegree $k(1+\alpha)$ 
by a suitable power of an invertible class (see \cite{Nitufib}) 
to represent them by a class $\hat{z}$ in 
bidegree $k(1-\lambda) + 0\alpha$. 
This procedure of normalizing yields a map of rings:
\begin{equation}
\label{comod}
\psi : [v_n]^{-1} BP_\ast [t_1, t_2, \ldots ] [v_n^{-1}] 
\longrightarrow  [\hat{v}_n^{-1}] 
\mathbb{BPR}_{(\ast, 0)} [\hat{t}_1, \hat{t}_2, 
\ldots ] [\hat{v}_n^{-1}]\, =  \, \pi_{(\ast,0)} 
(\mathbb{BP}(n) \wedge \mathbb{BP}(n)), 
\end{equation}
where the left hand side can be identified with the 
Hopf algebroid $BP(n)_\ast BP(n)$. 
Indeed, it is easy to see that $\psi$ is a 
map of Hopf algebroids  \cite[Theorem 4.11]{HK}, though we do not need it here.

The map $\psi$ that scales the grading by $(1-\lambda)/2$ comes
from a composition of maps.  First, it sends the 
generators $v_i \in \pi_\ast BP(n)$ 
to $\hat{v}_i \in \pi_{(\ast,0)} \mathbb{BP}(n)$. 
Then there is the canonical 
map $\pi_{(\ast,0)} \mathbb{BP}(n)  \rightarrow ER(n)_\ast$. 
This explains the unusual degrees in the map $\psi$ and
shows how the 
$ER(n)^*$ module spectral sequence,
$\mbox{E}_r^{*,*}(pt)$,
is also a spectral sequence of $E(n)^*$-modules.

Our goal is to show that each $\mbox{E}_r^{*,*}(pt)$
is also a $BP(n)_\ast BP(n)$-comodule
(equivalently, an $E(n)_*E(n)$-comodule)
that is finitely presented as a $BP(n)_\ast$-module
(equivalently, an $E(n)_*$-module)
and that each differential is a map in this category.

First, we have to describe the $E(n)^*$-module structure,
via $\psi$, of the spectral sequence.  This can be
confusing because the $E_1$ term of the spectral sequence
for a point is free over $E(n)^*$ on the $y^m$.  
At this point it is better to adopt new notation.
This $E(n)^*$-module structure via $\psi$ is equivalent
to being a module over 
$$
\hat{E}(n)^* = \Z_{(2)}[
\hat{v}_1, 
\hat{v}_2,
\ldots,
\hat{v}_n^{\pm 1}]
$$
by way of $\hat{E}(n)^* \rightarrow ER(n)^* \rightarrow E(n)^*$.
Now, 
making $E(n)^*$ an $\hat{E}(n)^*$-module is quite different
from the standard structure.
The primary difference is that 
$$
\psi(v_n) = \hat{v}_n = v_n v_n^{-(2^n-1)^2} = v_n^{-2^{n+1}(2^{n-1}-1)}.
$$
Consequently, $E(n)^*$, as a module over $\hat{E}(n)^*$,  is
a free module on generators 
\begin{equation}
\label{gen}
\{1,v_n,v_n^2,\ldots, v_n^{2^{n+1}(2^{n-1}-1)-1}\}.
\end{equation}
(The $n=1$ case is an anomaly because $\hat{v}_1 = 1$.  In this case
$E(1)^*$ is a free module over $\hat{E}(1)^*$ on an infinite
number of generators, $\{v_1^k\}$, $k \in \Z$.)
Note that we have, as we should have:
$$
\hat{E}(n)^* \otimes_{\hat{E}(n)^*} E(n)^* \simeq E(n)^*.
$$

Along the same lines as above, the rings, $R_n[v_n^{2^k}]$, $k \le n+1$, in
the description and proof of the spectral sequence for a point
are also free over $\hat{E}(n)^*$ on basis elements given by
$$
\{1,v_n^{2^k},v_n^{2(2^k)},v_n^{3(2^k)},\ldots, 
v_n^{(2^{n+1-k}(2^{n-1}-1)-1)(2^k)}\}.
$$
(For $n=1$ the basis is
over $v_1^{j2^k}$ for $j \in \Z$.)
Notice that by the time we are done with the spectral sequence and
we have $k=n+1$, we still have elements left for our basis:
$$
\{1,v_n^{2^{n+1}},v_n^{2(2^{n+1})},v_n^{3(2^{n+1})},
\ldots, v_n^{(2^{n-1}-2)(2^{n+1})}\}
$$
where our periodicity element is $v_n^{2^{n+1}}$, not $\hat{v}_n$,
and if you raise the periodicity element to the $-(2^{n-1}-1)$-th power,
you then get $\hat{v}_n$.
(For $n=1$ the periodicity element is $v_1^4$.  In the case $n=2$,
we do have the periodicity element $v_2^{{8}}= \hat{v}_2^{-1}$.)

We are now ready to put the $BP(n)_* BP(n)$-comodule structure on
the spectral sequence for a point.  We begin with the $\mbox{E}_1$
term, which is free on the $y^m$ over $E(n)^*$.  The element $y$
is odd degree so is unaffected by (co)operations, i.e. it is 
primitive.  The only comodule structure is on $\hat{E}(n)^*$ and only
goes on $E(n)^*$ through it.  $E(n)^*$ is free over $\hat{E}(n)^*$
on generators we have already given in Equation \eqref{gen}. 
As these generators have no connection to the operations, they are
also primitive.
All of the summands in the spectral sequence for
all $\mbox{E}_r(pt)$, sit on primitive elements $v_n^{i} y^m$.
The ideals, $I_j$ are known to be invariant
under the (co)operations, \cite{Land:Ann}, \cite{Land:Ass}, and
for our particular case, \cite{HS}.  

We need all of our module summands to be comodules and all of
our differentials to be maps of comodules.
We have shown that $\mbox{E}_1$ is in our category.  The
first differential is simply multiplication by 2 between
summands so is easily in our category.  Since the map $d_1$
is in the category, $\mbox{E}_2 = \mbox{E}_3$ is also in
our category.

Assume by induction that we have $\mbox{E}_{2^{k+1}-1}$ is
a comodule.  All we have to do is show that $d_{2^{k+1}-1}$
is a comodule map and we automatically have $\mbox{E}_{2^{k+1}}$
is a comodule and our induction will be complete.
However, our 
 $d_{2^{k+1}-1}$ takes a cyclic summand generator to $\hat{v}_k$
in a summand
in $R_n[v_n^{2^{k+1}}]/I_k$ (a free module over $\hat{E}(n)^*/I_k$).
In here, $\hat{v}_k$
is a primitive, 
 \cite{Land:Ann}, \cite{Land:Ass}, and \cite{HS},
so our differentials are in our
category of comodules.

We have now shown that 
 the entire Bockstein 
spectral sequence for a point belongs to the 
category of $BP(n)_\ast BP(n)$-comodules, equivalently, 
$E(n)_*E(n)$-comodules,
that are finitely presented as $BP(n)_\ast$-modules ($E(n)^*$-modules). 
We refer the reader to \cite{HS} for details on this category.

A Landweber flat $BP(n)_\ast$-module is 
defined to be a $BP(n)_\ast$-module that is 
flat on the category of finitely presented 
$BP(n)_\ast BP(n)$-comodules, \cite{Land:Hom} and \cite{HS}. 
Recall that by \cite[Theorem C]{HS}, we may as well 
work with $E(n)$ (as $\hat{E}(n)$) instead of $BP(n)$. 
This leads to the following:

\begin{thm}
\label{ss}
Suppose $M$ is Landweber flat $\hat{E}(n)^\ast$-module, 
and let $(\mbox{E}_\ast,d_\ast)$ denote the 
Bockstein spectral sequence for $X = pt$. 
Then $(M \otimes_{\hat{E}(n)^*} \mbox{E}_*,\id_M 
\otimes_{\hat{E}(n)^*} d_\ast)$ is a spectral sequence 
of $ER(n)^\ast$-modules that 
converges to $M \otimes_{\hat{E}(n)^*} ER(n)^*$.
\end{thm}

\section{$ER(n)^\ast(BO(q))$:}

In this section, we apply the results of previous sections to 
compute the $ER(n)$ cohomology of $BO(q)$. 
The key ingredients are the fact that $E(n)^*(BO(q))$ is 
Landweber flat, and that $\ER(n)$ is a 
Real oriented theory in the sense that the universal 
complex line bundle over $\C P^\infty$ admits an $\ER(n)$-theory 
Thom class. The standard Atiyah-Hirzebruch spectral 
sequence argument gives $\ER(n)$-theory Conner-Floyd 
Chern classes, which we will identify as permanent cycles in the 
Bockstein spectral sequence for $BO(q)$.

Let $\mathbb{ER}(n)$ denote the real oriented 
Johnson-Wilson theory. Let $BU(q)$ be endowed with the 
canonical involution given by complex conjugation. 
By the Atiyah-Hirzebruch spectral sequence we know that \cite{HK}
\[ \mathbb{ER}(n)^{\ast,\ast}(BU(q)) = 
\mathbb{ER}(n)^{\ast,\ast}[[{c}_1, \ldots, {c}_q]], \]
where ${c}_k$ is the (real) Conner-Floyd 
Chern classes in bidegree $k(1+\alpha)$. 
These classes lift the usual Chern classes $c_k$ under the 
forgetful map $\mathbb{ER}(n)^{a,b}(BU(q)) \longrightarrow 
E(n)^{a+b}(BU(q))$. 
Now recall from \cite{Nitufib}
that there is an invertible class $y(n) \in 
\mathbb{ER}(n)^{-\lambda,-1}$. The class ${c}_k \, y(n)^k$ 
is therefore a class in bidegree $k(1-\lambda) + 0\alpha$. 
Let $\hat{c}_k \in 
\mathbb{ER}(n)^{k(1-\lambda)}(BU(q))$ be this class. 
Notice that the $\hat{c}_k$ are lifts of the 
classes $c_k v_n^{k(2^n-1)} \in E(n)^*(BU(q))$. 
Now consider the restriction map along homotopy fixed points:
\[ \mathbb{ER}(n)^{(\ast, 0)}(BU(q)) 
\longrightarrow ER(n)^\ast(BO(q)). \]
This yields classes by the same name: $\hat{c}_k \in 
ER(n)^\ast(BO(q))$ that lift the restriction 
of $c_k v_n^{k(2^n-1)}$ to $E(n)^\ast(BO(q))$. 
In particular, we know that the classes $\hat{c}_k$ are 
permanent cycles in the Bockstein spectral sequence 
converging to $ER(n)^\ast(BO(q))$.

We may choose the classes $\hat{c}_k$ as replacements for the 
Chern classes and state the following old result 
(see \cite{WSW:BO}, 
\cite{KY}, \cite{Kriz}, and \cite{RWY}):

$
E(n)^\ast(BO(q)) \simeq  
$
$$
 E(n)^\ast[[c_1, \ldots, c_q]]/ 
( c_1 - c_1^\ast, \ldots,
 c_q - c_q^\ast )
= E(n)^\ast[[\hat{c}_1, 
\ldots, \hat{c}_q]]/ 
(\hat{c}_1 - \hat{c}_1^\ast , \ldots,
\hat{c}_q - \hat{c}_q^\ast ),
$$

where $\hat{c}_k^\ast$ are the conjugate 
Chern classes defined formally by the equality:
\[ \sum_{i=0}^\infty \hat{c}_i^\ast = \prod_{i=1}^\infty 
(1+ [-_{\hat F} x_i]), \quad \mbox{with} \quad 
\sigma_k(x_1, x_2, \ldots,) = \hat{c}_k\]
and with the 2-typical formal group law $\hat{F}$ 
defined over the ring $\Z_{(2)}[\hat{v}_1, \ldots ]$ via
\[ [2]_{\hat F}(u) = \sum_{i=0}^\infty{}^{\hat F} \hat{v}_i u^{2^i}. \]

\begin{defn}
Any $E(n)^*$-module, $M$, concentrated in even degrees, is 
also an $\hat{E}(n)^*$-module by the inclusion $\hat{E}(n)^*
\rightarrow E(n)^*$.
We define the $\hat{E}(n)^*$ module $\hat{M}$ to be all elements of
$M$ in degrees that are integral multiples of $2^{n+2}(2^{n-1}-1)$.
It is pretty crucial to note that the degree of $\hat{v}_k$ is
in such a degree since it is in degree $(2^k-1)2^{n+2}(2^{n-1}-1)$.
As a result, it will shortly turn out that $\hat{E}(n)^*$ is $\hat{M}$
for $M = E(n)^*$.
We need to account for all of the other elements in $M$, and
we claim that we have, following Equation \eqref{gen}, 
as an $\hat{E}(n)^*$-module,
$$
M \simeq \bigoplus_{j=0}^{2^{n+1}(2^{n-1}-1)-1} \hat{M} v_n^j.
$$
To see this, first observe that $v_n^{2^{n+1}(2^{n-1}-1)} = \hat{v}_n^{-1}$.
We want to show that the $v_n^j$ above hit every even degree 
from $0$ to $2^{n+2}(2^{n-1}-1)-1$, modulo 
$2^{n+2}(2^{n-1}-1)$.
It is enough to show that the only common divisor of the degree
of $v_n$, i.e.  $-2(2^n-1)$, and 
$2^{n+2}(2^{n-1}-1)$,
is $2$.
Divide both by $2$ and we see that $2$ does not divide $2^n-1$,
so this is just a question of $2^n-1$ and $2^{n-1}-1$ being
relatively prime.  We know that 
$(2^n-1) -
2(2^{n-1}-1)  
= 1$, so this is true.

For notational purposes,
in the case where our module is $E(n)^*(X)$,
we denote this by $\hat{E}(n)^*(X)$ without implying this is
a cohomology theory.

We now have 
$$
M \simeq \hat{M}\otimes_{\hat{E}(n)^*} E(n)^*,
$$
in particular, when $M = E(n)^*(X)$ or $E(n)^*$.
\end{defn}

\begin{remark}
\label{hatflat}
$\hat{E}(n)^*$ has a comodule structure on it from Equation \eqref{comod}.
It is identical to the usual comodule structure on $E(n)^*$.
If $M$ is Landweber flat as an $E(n)^*$-module, we claim that $\hat{M}$ is 
Landweber flat as an $\hat{E}(n)^*$-module.
Recall that Landweber flat just means that $v_k$ is injective
on $M/I_k M$ for all $k$.  In our case, because $v_n$ is a unit,
multiplication by $\hat{v}_k$ is equivalent to multiplication by
$v_k$ until we have $M/I_n M$.  This is just free over $K(n)^*
= \Z/2[v_n^{\pm 1}]$, which, in turn, is free over $\Z/2[\hat{v}_n^{\pm 1}]$.
Multiplication here by either $v_n$ or $\hat{v}_n$ is injective
and $M/I_{n+1} M = 0$, so if $M$ is Landweber flat over $E(n)^*$,
$\hat{M}$ is Landweber flat over $\hat{E}(n)^*$ just by restricting
to the degrees that define $\hat{M}$.
\end{remark}

We are now ready to prove Theorem \ref{main}.

\begin{proof}

Our subset $\hat{E}(n)^*(BO(q))$  of $E(n)^*(BO(q))$ consists
of
permanent cycles in the
Bockstein spectral sequence. 
We may re-express $E(n)^\ast(BO(q))$ as:
\begin{equation}
\label{ebo}
E(n)^\ast(BO(q)) \simeq (\hat{E}(n)^\ast[[\hat{c}_1, 
\ldots, \hat{c}_q]]/
(\hat{c}_1 - \hat{c}_1^\ast , \ldots,
\hat{c}_q - \hat{c}_q^\ast )
) \otimes_{\hat{E}(n)^*} E(n)^\ast. 
\end{equation}

Writing $E(n)^\ast(BO(q))$ as a tensor product in this 
manner has the advantage of expressing it in terms of 
permanent cycles in the Bockstein spectral sequence 
for $X = BO(q)$, extended by the coefficients $E(n)^\ast$. 
We know that 
$E(n)^\ast[[\hat{c}_1, \ldots, 
\hat{c}_q]]/
(\hat{c}_1 - \hat{c}_1^\ast , \ldots,
\hat{c}_q - \hat{c}_q^\ast )
$ 
is 
Landweber flat, since it can be identified with $E(n)^\ast(BO(q))$. 
This was first stated explicitly in \cite{KY}.
It also follows by combining \cite{Kitch-Wil-KnBOq} and \cite{RWY}.
By Remark \ref{hatflat}, we see that
$\hat{E}(n)^\ast[[\hat{c}_1, \ldots, 
\hat{c}_q]]/
(\hat{c}_1 - \hat{c}_1^\ast , \ldots,
\hat{c}_q - \hat{c}_q^\ast )
$ 
is also Landweber flat.
Invoking Theorem \ref{ss}, 
one has a canonical map of spectral sequences:
$$
 \varphi : (\hat{E}(n)^\ast[[\hat{c}_1, \ldots, \hat{c}_q]]/
(\hat{c}_1 - \hat{c}_1^\ast , \ldots,
\hat{c}_q - \hat{c}_q^\ast ))
\otimes_{\hat{E}(n)^*} (\mbox{E}_\ast, d_\ast) 
\longrightarrow (\mbox{E}_\ast(BO(q)), d_\ast), 
$$
where, as mentioned earlier, the left hand side is the 
Bockstein spectral sequence for a point tensored with the 
ring of permanent cycles: 
$\hat{E}(n)^\ast[[\hat{c}_1, 
\ldots, \hat{c}_q]]/
(\hat{c}_1 - \hat{c}_1^\ast , \ldots,
\hat{c}_q - \hat{c}_q^\ast )
$, 
and the right hand side is the 
Bockstein spectral sequence for $X = BO(q)$. 

Since the map $\varphi$ is manifestly an isomorphism 
at the $\mbox{E}_1$-term by Equation \eqref{ebo}, we conclude that 

$ER(n)^\ast(BO(q)) \simeq$
\begin{multline*}
\qquad (\hat{E}(n)^\ast[[\hat{c}_1, \ldots, 
\hat{c}_q]]/
(\hat{c}_1 - \hat{c}_1^\ast  , \ldots,
\hat{c}_q - \hat{c}_q^\ast  ))
\otimes_{\hat{E}(n)^*} ER(n)^\ast \simeq \\
ER(n)^\ast[[\hat{c}_1, \ldots, 
\hat{c}_k]]/ 
(\hat{c}_1 - \hat{c}_1^\ast  , \ldots,
\hat{c}_q - \hat{c}_q^\ast  ).
\qquad
\end{multline*}
\end{proof}

\section{Orientation of Bundles:}

\noindent
Let $\xi : V \rightarrow B$ be a real vector bundle that 
supports an $E(n)$-orientation. To see if $\xi$ is 
orientable in $ER(n)$, we may run the Bockstein spectral 
sequence on the Thom space $B^\xi$, and ask if there is a 
Thom class in the $\mbox{E}_1(B^\xi)$-term that survives 
to $\mbox{E}_{2^{n+1}}(B^\xi)$.

Let us work in the universal (virtual) bundle $\xi$ over $B = BO$. 
This bundle is not $E(n)$-orientable. 
However, twice it is indeed $E(n)$-orientable 
since $2\xi = \xi \otimes \C$. Another way to say this is that 
given the multiplication by $2$- map $[2] : BO \longrightarrow 
BO$, the pullback bundle $[2]^\ast \xi$ is $E(n)$-orientable. 

\begin{defn}
Define the bundle $2^k \xi$ over $BO$ to be the pullback of the 
universal bundle $\xi$ along the multiplication 
by $2^k$-map $[2^k] : BO \longrightarrow BO$. Let $MO[2^k]$ 
denote the Thom spectrum of the bundle $2^k\xi$.
\end{defn}

Let us prove an elementary result about $MO[2]$. 

\begin{prop}
The Thom spectrum $MO[2]$ supports an $E(n)$-orientation. 
Furthermore, it supports a Thom class $\mu \in E(n)^0(MO[2])$ 
with the property that $c(\mu) = \mu$, where $c$ is the 
involution induced by complex conjugation on $E(n)$. 
\end{prop}

\begin{proof}
Recall from earlier remarks that $MO[2]$ can be identified with 
Thom spectrum of the virtual (complex) vector bundle given by 
restricting the universal bundle over $BU$ along $BO$. 
In particular, one has a map of Thom spectra $i : 
MO[2] \longrightarrow MU$. Since $MU$ supports a 
canonical $E(n)$-orientation given by a 
Thom class $\mu \in E(n)^0(MU)$, one obtains a 
Thom class by restriction (and denoted by the same 
symbol) $\mu \in E(n)^0(MO[2])$. 
It remains to show that $c(\mu) = \mu$ in $E(n)^0(MO[2])$. 
We proceed as follows:

Consider the real orientation of $E(n)$ 
given by a $\Z/2$-equivariant map: 
\begin{equation} 
\label{leftright}
\mu_1 : MU(1) \longrightarrow \Sigma^{(1+\alpha)} E(n). 
\end{equation} 
The action of complex conjugation on $\Sigma^{1+\alpha} 
E(n)$ can be identified with $-c$ (the $c$ from $E(n)$
and the $-1$ from the orientation reversing action on the
two sphere).

On the other hand, complex conjugation on $MU(1)$ is induced 
by the (complex anti-linear) self map of the universal line 
bundle $\gamma_1$ over $BU(1)$ that sends a vector to its 
complex conjugate. This map can be seen as a (complex linear) 
isomorphism from $\overline{\gamma_1}$ to $\gamma_1$, 
where $\overline{\gamma_1}$ is the opposite complex 
structure on the real bundle underlying $\gamma_1$. 
Since $\overline{\gamma}_1$ is isomorphic to the dual 
bundle $\gamma_1^*$, we see that the action of complex 
conjugation on $MU(1)$ sends the Thom class $\mu_1 \in 
E(n)^2(MU(1))$ to the class $[-_F \, \mu_1]$, where $F$ is the 
formal group law for $E(n)$. 
$MU(1)$ and $BU(1)$ are homotopy equivalent and the Thom
isomorphism is an $E(n)^*(BU(1))$ module map so
that $\mu_1^2 = \mu_1 \, x$, 
where $x = c_1(\gamma_1)$. 

From this we have $\mu_1^k = \mu_1 x^{k-1}$, so if we have
any power series  $\sum a_i \mu_1^{i+1}$,
we can rewrite this as $\mu_1 \, \sum a_i x^i$.
Hence $[-_F  \,\mu_1]$ = $\mu_1 \frac{[-_F \, x]}{x}$.

Incorporating this observation into the $\Z/2$ 
equivariance of $\mu_1$ from Equation \eqref{leftright}
by computing on the left and the right, this
translates to the equality: 
$$
\mu_1 \frac{[-_F  \,x]}{x} 
= 
 - c(\mu_1). 
$$

Let $\mu_1^{2k} : MU(1)^{\wedge 2k} \longrightarrow 
\Sigma^{2k(1+\alpha)} E(n)$ denote the external smash 
product of $2k$-copies of $\mu_1$. It follows that:
\[ c(\mu_1^{2k}) = \mu_1^{2k} \prod_{i=1}^{2k} 
\frac{[-_F \, x_i]}{x_i} = \mu_1^{2k} \, 
\frac{\sigma_{2k}(-_F \, x_1, -_F \, x_2, 
\ldots, -_F \, x_{2k})}{\sigma_{2k}(x_1, x_2, \ldots, x_{2k})}, \]
where $\sigma_{2k}$ denotes the $2k$-th elementary 
symmetric function in $2k$-variables. It follows 
from this that the canonical orientation $\mu_{2k} : 
MU(2k) \longrightarrow \Sigma^{2k(1+\alpha)} E(n)$ has the property:
\[ c(\mu_{2k}) = \mu_{2k} \, \frac{c_{2k}^*}{c_{2k}}, \]
where $c_{i}$ are the Conner-Floyd Chern classes with 
conjugates denoted by $c_i^*$. Now let $MO(2k)[2]$ denote the 
Thom spectrum of the bundle over $BO(2k)$ given by 
restriction along $BO(2k) \longrightarrow BU(2k)$. 
Let $\mu_{2k} \in E(n)^{4k}(MO(2k)[2])$ also 
denote the restriction of $\mu_{2k} \in E(n)^{4k}(MU(2k))$. 
Since $c_i = c_i^*$ in $E(n)^*(BO(2k))$, we 
observe that $c(\mu_{2k}) = \mu_{2k}$ in $E(n)^{4k}(MO(2k)[2])$. 
Taking inverse limits over $k$, we deduce 
that $c(\mu) = \mu$ in $E(n)^0(MO[2])$. 
\end{proof}

We can now prove the first part of Theorem \ref{orient}.

\begin{proof}

The spectral sequence for a Thom space is a module over the
spectral sequence for $BO$.  The only new concern is what
happens with the Thom class.  If $d_r = 0$ on the Thom class,
then $\mbox{E}_{r+1}$ for the Thom space
is a rank one free module over $\mbox{E}_{r+1}(BO)$ on the Thom class.

Consider the Bockstein spectral sequence $\mbox{E}_\ast(MO[2])$ 
converging to $ER(n)^\ast(MO[2])$. 
By the previous Proposition, we see that $\mbox{E}_1(MO[2])$ is a rank one 
module over $\mbox{E}_1(BO)$ generated by the complex orientation $\mu$
for $MO[2]$. Furthermore, we know that $c(\mu) = \mu$. 
It follows that $d_1(\mu)= y \, v_n^{-(2^n-1)}(1-c)(\mu) = 0$. 
This gives us $\mbox{E}_2$, but to begin our induction, we
need to have $d_2 = 0 $ on the Thom class.
$d_2 (\mu)$ has to hit some element $zy^2$ 
with $z \in \mbox{E}_2^{0,p}(BO)$ where $p=2\lambda+1 = 2^{2n+2}-2^{n+2}+5$,
but there are no odd degree elements.  We have 
$\mbox{E}_3(MO[2])$ is a rank one 
module over $\mbox{E}_3(BO)$ generated by the complex orientation $\mu$.

Assume by induction that for $k > 1$, $\mbox{E}_{2^{k}-1}(MO[2^{k-1}])$ is a 
rank one module over $\mbox{E}_{2^{k}-1}(BO)$ on a distinguished 
generator $u_{k-1}$. 
Writing $2^{k}\xi$ as $2(2^{k-1}\xi)$, we have a factorization:
\[ [2^{k}] = ([2^{k-1}] + [2^{k-1}]) \circ \Delta : 
BO \longrightarrow BO \times BO \longrightarrow BO, \]
which induces a map of Thom spectra:
\[ \tau = \mu \circ \Delta : MO[2^{k}] \longrightarrow MO[2^{k-1}] 
\wedge MO[2^{k-1}] \longrightarrow MO[2^{k-1}]. \]
This induces a map of spectral sequences $\tau^\ast : 
\mbox{E}_{2^{k}-1}(MO[2^{k-1}]) \longrightarrow 
\mbox{E}_{2^{k}-1}(MO[2^{k}])$. 
Define $u_{k}$ to be $\tau^\ast(u_{k-1})$. 
Observe that $c(u_k)=u_k$ by naturality and induction from $u_{k-1}$.
This is needed for 
the following equation using Theorem \ref{bss}, {\em (vi)}:
\[ d_{2^{k}-1}(u_{k}) = \Delta^\ast \mu^\ast 
d_{2^{k}-1}(u_k) = \Delta^\ast d_{2^{k}-1} \mu^\ast (u_k) = 
\Delta^\ast (d_{2^{k}-1}(u_{k-1}) \wedge u_{k-1} + u_{k-1} 
\wedge d_{2^{k}-1}(u_{k-1})). \]
But notice that $\Delta : MO[2^{k}] 
\longrightarrow MO[2^{k-1}] \wedge MO[2^{k-1}]$ is invariant 
under the swap map on $MO[2^{k-1}] \wedge MO[2^{k-1}]$. 
Therefore $\Delta^\ast (d_{2^{k}-1}(u_{k-1}) \wedge u_{k-1}) = 
\Delta^\ast (u_{k-1} \wedge d_{2^{k}-1}(u_{k-1}))$. 
It follows that $d_{2^{k}-1}(u_{k})$ is a 
multiple of $2$ and must consequently be zero 
since $\mbox{E}_{2^{k}-1}(MO[2^{k}])$ is a $\Z/2$-module 
for external degrees greater than zero \footnote{The next claim shows that $d_{2^k-1}(u_{k-1})$ is in fact non-trivial}. It follows that $u_{k}$ survives 
to $\mbox{E}_{2^{k}}(MO[2^{k}])$.

We need to show that $d_{2^{k}+r}(u_{k})=0$ by induction
on $r$, for $0 \le r < 2^{k}-1$ to complete the induction
on $k$.
$d_{2^{k}+r}(u_{k})=z y^{2^{k}+r}$ by induction
where $|z| = (2^k + r) \lambda + 1$, $z \in \mbox{E}_{2^k+r}^{0,|z|}$
(from Theorem \ref{bss} {\em (ii)}).
Now by the flatness of $E(n)^\ast(BO)$, we have the 
Thom isomorphism:
$$
\mbox{E}_{2^{k}+r}(MO[2^{k}]) = \mbox{E}_{2^{k}+r}(pt) 
\otimes_{\hat{E}(n)^\ast} \hat{E}(n)^\ast(BO). 
$$
From Corollary \ref{cor}, Theorems \ref{structure} and \ref{ss}, 
we know the only 
elements to survive to sit on 
$ y^{2^{k}+r}$ are elements from $BO$ and $R_n/I_k$ sitting on
$v_n^{a2^k}$.  The degrees of the $c_k$ are divisible by $2^{n+2}$
and the degree of elements in $R_n/I_k$ sitting on $v_n^{a2^k}$ are
divisible by $2^{k+1}$ because $v_n^{a2^k}$ is and every element
in $R_n/I_k$ is made up of multiples of elements $\hat{v}_j$ with
$j \ge k$, and all of them have degree divisible by $2^{n+2}$.  
Modulo
$2^{k+1}$, $|z| = 2^k + r + 1$, which is less than $2^{k+1} $,
so $z$ must be zero as it is not in a degree with any elements.

The proof of the first part of Theorem \ref{orient} is 
complete on observing that the 
spectral sequence collapses at $E_{2^{n+1}}$. 
\end{proof}

The last part of Theorem \ref{orient} is complete with the following.

\begin{claim}
There exists a vector bundle $\zeta$ such that $2^n \zeta$ is not $ER(n)$-orientable. In particular, the last possible differential $d_{2^{n+1}-1}$ in the Bockstein spectral sequence converging to $ER(n)^*(MO[2^n])$ is nontrivial on the generator $u_n$. 
\end{claim}
\begin{proof}
Let $S^{\alpha}$ denote the one point compactification of the sign representation of $\Z/2$. Consider the virtual vector bundle $\zeta$ over $B\Z/2$ with Thom space given by: 
\[ Th(\zeta) = E\Z/2_+ \wedge_{\Z/2} S^{(\alpha-1)}. \] 
If $2^n\zeta$ were to admit an $ER(n)$-orientation, then one obtains a $\Z/2$-equivariant map representing the Thom class $\mu$: 
\[  \mu : E\Z/2_+ \wedge S^{2^n(\alpha-1)} \longrightarrow ER(n) \longrightarrow E(n). \]
By adjointness, one obtains a class $\mu \in \pi_{2^n(\alpha-1)} \E\R(n)$, where $\E\R (n)$ denotes the $\mbox{RO}(\Z/2)$-graded real spectrum representing Johnson-Wilson theory \cite{Nitufib}. This class has the property that it restricts to a unit on forgetting the $\Z/2$-equivariant structure. However, the computation of the bigraded homotopy of $\E\R (n)$ given in \cite{HK} shows that there is no such class. Hence we obtain a contradiction to the existance of $\mu$. 
\end{proof}

\begin{remark}
Theorem \ref{orient} is by no means optimal. 
For example, for $n=1$ we know that $ER(1)$ is 2-localized real 
$K$-theory. Hence a bundle $\zeta$ is $ER(1)$-orientable if and 
only if it is {\em spin}. This is equivalent to $w_1(\zeta) = w_2(\zeta) = 0$. 
This holds for bundles of the form $\zeta = 4\xi$, 
but clearly there are spin bundles that are not divisible by $4$. Similarly, for $ER(2)$, the results of \cite{KS} suggest that a bundle $\zeta$ is $ER(2)$-orientable if and only if $w_1(\zeta) = w_2(\zeta) = w_4(\zeta) = 0$, which is clearly true for bundles of the form $\zeta = 8\xi$. 
It is a compelling question to find a nice answer in 
general for when a bundle is $ER(n)$-orientable, or even to 
show that an answer to this question may be given in closed form. 
\end{remark}


\pagestyle{empty}

\end{document}